\documentclass[12pt,a4paper]{article}
\usepackage{amsmath, amssymb, amsthm, verbatim} 
\usepackage{graphicx}
\usepackage{pxfonts} 

\usepackage{bm}
\usepackage[all]{xy}
\usepackage{eucal} 
\theoremstyle{plain}
  \newtheorem{thm}{Theorem}[section]
  \newtheorem{prop}[thm]{Proposition}
  \newtheorem{lemma}[thm]{Lemma}

  \newtheorem{rem}[thm]{Remark}

\newcommand{\INT}[1]{{\overset{\circ}{#1}}}

\begin{document} 
\title{The Elliptic Hypergeometric Functions Associated to the Configuration 
Space of Points on an Elliptic Curve I : Twisted Cycles} 
\author{Ko-Ki Ito}   
\maketitle

\abstract{We consider the Euler type integral associated to the configuration space of points on an elliptic curve, 
which is an analogue of the hypergeometric function associated to the configuration space of points on a 
projective line.  
We calculate the {\it twisted homology group}, 
with coefficients in the local system associated to a power function $g^{\alpha}$ of an elliptic function $g$, 
and the intersection form. 
Applying these calculations, we describe the {\it connection matrices} representing the 
linear isomorphisms induced 
from analytic continuations of the functions defined by the integrations of $g^{\alpha}$ 
over twisted cycles.}



\section{Introduction} 
In many cases, 
hypergeometric functions 
have (or are defined by) 
Euler type integrals. 
In fact, the hypergeometric function associated to the configuration space of 
points on $\Bbb{P} ^1$ has 
the following integral representation: 
\begin{align*} 
 \int _{x_i} ^{x_j} ( t - x_1) ^{\alpha _1} (t-x_2 ) ^{\alpha _2} \cdots (t-x_n ) ^{\alpha _n} dt . 
\end{align*} 
In this paper, 
instead of $\Bbb{P} ^1$, 
we shall consider the configuration space of points on an elliptic curve.   

Let $C$ be an elliptic curve defined by 
$\Bbb{C} / \Gamma$, 
where 
$\Gamma := \Bbb{Z} \omega _1 \oplus \Bbb{Z} \omega _2$ 
is a lattice generated by $\omega _1$ and $\omega _2$. 
We shall denote a point on $C$ by an equivalence class $\overline{x} \in C$ 
 represented by $x \in \Bbb{C}$. 
For a triple 
$q= ( \overline{x_0} , \overline{x_1} , \overline{x_2} )$ of points on $C$, 
we denote by  
$X_q$ by the punctured elliptic curve, 
that is,  
the open Riemann surface deprived of four points  
( $ \overline{x_0} , \overline{x_1} , \overline{x_2} , \overline{0} \in C$).  
Our main objective is the integration
of  the multi-valued function $g ^{\alpha}$, 
where 
$\alpha \in \Bbb{C} \setminus  
\left( \frac{1}{2} \Bbb{Z} \cup \frac{1}{3} \Bbb{Z} \right)$ 
is a fixed complex number,  
whose branch points are 
$ \overline{x_0} , \overline{x_1} , \overline{x_2} , \overline{0} \in C$. 
The function $g$ is expressed by Weierstra{\ss} $\sigma$ function: 
\begin{align*} 
g 
=  
\frac{\sigma (t-x_0 ) \sigma (t-x_1) \sigma (t-x_2) }
{\sigma ^3 (t) \sigma (x_0) \sigma (x_1) \sigma (x_2)} .  
\end{align*}  
It depends only on the equivalence class $\overline{t}$. 
So this function $g$ is a meromorphic function on $C$ holomorphic on $X_q$, 
which is a counterpart of rational functions on $\Bbb{P} ^1$. 
Thus the function $g^{\alpha}$ can be seen as an analogue of 
the multi-valued function 
$t ^{-3\alpha} (t -x_0 ) ^{\alpha} (t-x_1 )^{\alpha} (t-x_2) ^{\alpha}$ 
on $\Bbb{P} ^1$.     
However, it may depend on a representative $x_0$ 
(resp. $x_1$, $x_2$) of the equivalence 
class $\overline{x_0}$ (resp. $\overline{x_1}$, $\overline{x_2}$). 
To avoid this, we assume that $x_0 +x_1 +x_2 =0$.   
The integration of $g^{\alpha} dt$ over a path $\Xi _{\mu ,q}$ 
(mentioned below) lying on $X_q$ 
\begin{align} \label{int0} 
F _{\mu} (q) := \int _{\Xi _{\mu ,q}} g^{\alpha} dt   
\end{align} 
is a multi-valued function with respect to $q$ defined over the {\it configuration space} S 
(defined later), which is a subset of  
$\overline{S} := \left\{  q = ( \overline{x_0} , \overline{x_1} , \overline{x_2} )  \ 
\right| \left.  x_0 + x_1 + x_2 =0   \right\} \subset  C \times C \times C $.  

This function $F_{\mu}$ has {\it singularities} : 
Put 
$D ^{ij} := \left\{  q = ( \overline{x_0} , \overline{x_1} , \overline{x_2} )  \in \overline{S} \ 
\right| \left.  x_i -x_j \in \Gamma    \right\}$ and  
$D ^{i} _{\infty} :=  \big\{ 
q = ( \overline{x_0} , \overline{x_1} , \overline{x_2} )  \in \overline{S} \ 
\big| 
 \   x_i  \in \Gamma    
\big\} 
$; 
when $q$ belongs to $D^{ij}$, 
the topological type of $X_q$ differs from that of 
an punctured Riemann surface deprived of four points from $C$ ; 
when $q$ belongs to $D^i _{\infty}$, 
the integrand $g^{\alpha}$ is divergent.  
Hence we define a domain $S$ of the function $F_{\mu}$ 
by removing the singular loci $D^{ij}$ and $D^i _{\infty}$ from $\overline{S}$. 

We introduce an integration path $\Xi _{\mu ,q}$ 
of ({\ref{int0}}), the so-called {\it twisted cycle}. 
Aomoto theory (\cite{aomoto}) tells us that 
a path of an Euler type integral can be regarded as a homology class 
with coefficients in the local system defined by the multi-valuedness of the integrand. 
In our case, 
the local system $\mathcal{L} _{X_q}$ is defined by the multi-valuedness of $g^{\alpha}$ 
and we denote by  $\{ \Xi _{\mu ,q} \} _{\mu}$ generators of $H_1 ( X_q , \mathcal{L} _{X_q} )$.   
In {\S}{\ref{sec_twisted_cycles}}, we give the concrete description of $\{ \Xi _{\mu ,q} \} _{\mu}$  
(Theorem {\ref{thm_twisted_cycles}}). 
In {\S}{\ref{sec_intersection_form}}, we calculate the intersection form among the twisted cycles 
(Theorem {\ref{intersection_form}}). 

The function $F_{\mu}$ on $S$ can be seen as a flat section of a certain vector bundle over $S$  
or, equivalently, 
a solution of a certain linear differential equation({\cite{ito}}) :  
The sheaf $\mathcal{S}$  of germs of functions defined by 
$\Bbb{C}$-linear combinations of $\{ F_{\mu} \} _{\mu}$ 
is locally constant and $\mathcal{S} \otimes \mathcal{O} _S$ 
is a flat (trivial) vector bundle over $S$. 
The analytic continuation along a path $\gamma$ in $S$ induces the linear isomorphism 
between the stalks of $\mathcal{S}$ over the initial point  and  the terminal point of $\gamma$ ;   
the corresponding matrix is said to be the {\it connection matrix}. 
In {\S}{\ref{sec_connection_matrix}}, 
we consider the paths in $S$ connecting two points corresponding to a configuration 
$\left( \overline{\frac{\omega _i}{2}} , \overline{\frac{\omega _j}{2}} ,\overline{\frac{\omega _k}{2}} \right)$ 
and another one given by interchanging two of the three points. 
For these paths, we calculate the connection matrices (Theorem {\ref{thm_connection_matrix}}). 

Finally, we briefly mention the relation with the recent work of H.Watanabe({\cite{watanabe1}}{\cite{watanabe2}}). 
He considers there a {\it fixed} configuration ; the two-torsion points are removed from elliptic curves, 
which are deformed by the moduli parameter. 
He focuses the action of modular transformations on such elliptic curves and 
studies that action on (co)homology groups of elliptic curves deprived of the two-torsion points 
with coefficients in certain local systems, 
which are of different kind of ours. 
In this context, 
he interprets the monodromy transformations of Gau{\ss}'s hypergeometric function into 
the modular transformations. \\ 

{\textbf{Acknowledgments}}.    
The author would like to thank H.Watanabe and T.Mano for fruitful discussions. 
The author is sincerely grateful to K.Takamatsu for his guidance.
The author would like to thank also Professors F.Kato and K.Fukaya 
for valuable comments and encouragements.

\section{The local system} \label{sec.loc} 
For a triple $q=\left( \overline{x_0} , \overline{x_1} , \overline{x_2} \right)$ 
of points on $C= \Bbb{C} / \Gamma$, 
we denote by $X_q$ an open Riemann surface deprived of four points 
$\overline{x_0} , \overline{x_1} , \overline{x_2} , \overline{0}$ 
from $C$. 
We assume that $x_0 + x_1 + x_2 =0$ throughout this paper. 
We shall introduce a local system 
$\mathcal{L} _{X_q}$ which reflects the data of multi-valuedness of the function 
$g ^{\alpha}$ defined as follows. 
Let $g$ be the following holomorphic function on $X_q$: 
\begin{align*} 
g:= \frac{\sigma (t-x_0 ) \sigma (t-x_1) \sigma (t-x_2) }
{\sigma ^3 (t) \sigma (x_0) \sigma (x_1) \sigma (x_2)} . 
\end{align*}  
Using the identity $\frac{\sigma ^{\prime} (t)}{\sigma (t) }=\zeta (t)$, we have  
$\frac{dg}{g} 
=( \zeta ( t-x_0 ) + \zeta (t-x_1 ) +\zeta (t-x_2 ) -3\zeta (t) ) dt$. 
We shall consider the following sheaf 
$\mathcal{L} _{X_q}$: 
\begin{align*} 
\mathcal{L} _{X_q}  := 
\mathrm{Ker} 
\left( 
d - \alpha \frac{dg}{g} : 
\mathcal{O} _{X_q} 
\longrightarrow 
\Omega ^1 _{X_q} 
\right) ,  
\end{align*} 
where $\alpha \in \Bbb{C} \setminus  
\left( \frac{1}{2} \Bbb{Z} \cup \frac{1}{3} \Bbb{Z} \right)$ 
is a fixed complex number. 
This sheaf has no global section on $X_q$. 
We take the open part $\INT{X_q}$ in $X_q$ such that 
$\mathcal{L} _{X_q}$ has global sections on $\INT{X_q}$: 
\begin{align*} 
\INT{X_q} := C \setminus 
\left( \overline{0 x_0} \cup \overline{0 x_1} \cup \overline{0 x_2} \right) ,  
\end{align*} 
where $\overline{0 x_i}$ is the image of 
the segment $[0,x_i]$ in $\Bbb{C}$ through the natural projection 
$\Bbb{C} \longrightarrow C$. 
\begin{lemma}  \label{lemma3.1} 
\begin{enumerate} 
\item The sheaf 
$\mathcal{L} _{X_q}$ is a local system of rank $1$ over $\Bbb{C}$
 whose local sections are generated by a branch of $g^{\alpha}$. 
\item The sheaf $\mathcal{L} _{X_q}$ has a non-zero global section 
$\varsigma$ on $\INT{X_q}$. 
\end{enumerate} 
\end{lemma} 
\begin{proof} 
\begin{enumerate} 
\item 
The function $g^{\alpha}$ is one of the solutions of 
the first-order ordinary linear differential equation 
$\left( d - \alpha \frac{dg}{g} \right) \varsigma =0$. 
By the local theory of ordinary differential equations, 
local solutions of a first order linear differential equation are generated by 
a non-zero solution,
whence the assertion. 
\item 
We shall prove that a single-valued solution of 
the equation $\left( d - \alpha \frac{dg}{g} \right) \varsigma =0$ 
on $\INT{X_q}$ does exist. 
The logarithm of its solution 
 is given by an integration of $\alpha \frac{dg}{g}$ over an arc. 
The single-valuedness is equivalent to the condition that 
the integration over any closed loop vanishes.   
Note that $\INT{X_q}$ is homotopic to $S^1 \vee S^1$. 
In fact, we can pick two loops $l_{\omega _1}$, $l_{\omega _2}$ on $\INT{X_q}$ 
such that the inclusion map 
$l_{\omega _1} \cup l_{\omega _2} \hookrightarrow \INT{X_q}$ 
is a deformation retract. 
That is the reason why we verify 
the integrations of $\frac{dg}{g}$ over those two loops $l_{\omega _1}$, 
$l_{\omega _2}$ vanish. 
We denote by $I_i (x_0 ,x_1 ,x_2 )$ these integrations: 
\begin{align*} 
I_i (x_0 , x_1 , x_2 ) := \int _{l_{\omega _i}} \frac{dg}{g} 
= \int _{l_{\omega _i}}
 ( \zeta ( t-x_0 ) + \zeta (t-x_1 ) +\zeta (t-x_2 ) -3\zeta (t) ) dt . 
\end{align*}   
Using the identity $\zeta ^{\prime} (t) =-\wp (t)$, 
we have 
$\frac{\partial I _i}{\partial x_j} = \int _{l_{\omega _i}} \wp ( t-x_j) dt
= \int _{l_{\omega _i}} \wp ( t ) dt$. 
This is a constant, which we denote by $\pi _i$.  
Then the equality 
$I _i (x_0 ,x_1 ,x_2 ) = \pi _i x_0 + \pi _i x_1 + \pi _i x_2 + I_i (0,0,0)$ holds. 
Note that $x_0 + x_1 +x_2 =0$ (by the assumption) and 
$I_i (0,0,0) = 0$ (by its definition). 
Hence the integrations $I_i (x_0 ,x_1 ,x_2)$ vanish.  
\end{enumerate} 
\end{proof} 
Now we pick a non-zero global section $\varsigma$ over $\INT{X_q}$ 
and fix it once for all in the sequel. 
For a path $l : I \longrightarrow X_q$ with the initial point  
in $\INT{X_q}$, 
we denote by $\varsigma _l$ the analytic continuation 
of $\varsigma$ along $l$. 
\begin{rem} \label{rem3.2} 
Note that $X_q$ is homotopic to 
the bouquet $\underbrace{S^1 \vee \cdots \vee S^1}_{5 \mathrm{copies}}$. 
In fact, we can pick five loops $l_0$, $l_1$, $l_2$, $l_{\omega _1}$, 
$l_{\omega _2}$ such that 
the inclusion map 
$l_0 \cup l_1 \cup l_2 \cup l_{\omega _1} \cup l_{\omega _2} 
\hookrightarrow X_q$ is a deformation retract. 
The loops $l_{\omega _1}$, $l_{\omega _2}$ are the same ones 
as in  the proof of lemma {\ref{lemma3.1}}. 
The loop $l_i$ is a closed curve  around the point $\overline{x_i}$.  
The representation $\rho : \pi _1 (X_q ) \longrightarrow \Bbb{C} ^{\ast}$ 
corresponding to 
the local system $\mathcal{L} _{X_q}$ is given by the formulae  
$\rho ( l_{i} ) = c$, $\rho ( l _{\omega _i} ) = 1$, 
where $c:=e^{2 \pi \sqrt{-1} \alpha}$.  
\end{rem}

\section{The twisted homology} \label{sec_twisted_cycles}
We shall find  generators of  the {\it twisted homology group} 
$H_1 ( X_q , \mathcal{L} _{X_q} )$.  

The $k$-chain group 
$C_k ( X_q , \mathcal{L} _{X_q} )$ 
with coefficients in $\mathcal{L} _{X_q}$  
is the complex vector space with basis 
$\sigma \otimes s _{\sigma}$, 
where $\sigma : \Delta \longrightarrow X_q$ is a singular $k$-simplex 
and $s _{\sigma}$ is a section of $\mathcal{L} _{X_q}$ over $\sigma$. 
The first boundary operator  
$\partial : C_1 ( X_q , \mathcal{L} _{X_q} ) \longrightarrow 
C_0 ( X_q , \mathcal{L} _{X_q} ) $ is given by 
$\partial ( l \otimes s_l ) 
= l(1) \otimes s _{l ,l(1) } - l(0) \otimes s_{l , l(0)}$.  
The other boundary operators are given in similar fashions.  
The 1-st {\it twisted homology group} 
$H_1 ( X_q , \mathcal{L} _{X_q} )$ 
is 
the 1-st homology of this complex, 
whose elements are called the {\it twisted cycles}. 
For two loops $l_{\omega _1}$, $l_{\omega _2}$ 
in the proof of lemma {\ref{lemma3.1}}, 
we denote by  $\Xi _{\omega _1 , q}$ 
(resp. $\Xi _{\omega _2 , q}$) 
the twisted cycle  
$ l_{\omega _1} \otimes \varsigma _{l_{\omega _1 }}$ 
(resp. 
$ l_{\omega _2} \otimes \varsigma _{l_{\omega _2 }}$),   
where $\varsigma$ is given in {\S}{\ref{sec.loc}}.  
We introduce  other twisted cycles $\Xi _{(ij) ,q}$,  
which are   the {\it regularizations} 
of the segments $\overline{x_i x_j}$:  
\begin{align*} 
\Xi _{(ij) ,q} 
:= 
\frac{1}{c -1 } \overline{S _{\epsilon} (x_i)} 
\otimes 
\varsigma _{\overline{S _{\epsilon} (x_i)} } 
+ \overline{[ x_i +\epsilon , x_j -\epsilon ]} 
\otimes 
\varsigma _{\overline{[ x_i +\epsilon , x_j -\epsilon ]}} 
- \frac{1}{c -1 } \overline{S _{\epsilon} (x_j)} 
\otimes 
\varsigma _{\overline{S _{\epsilon} (x_j)} } , 
\end{align*} 
where $\epsilon$ is a sufficiently small positive number; 
we denote by 
$[ x_i +\epsilon , x_j -\epsilon ]$ 
the oriented segment given by truncating each of both ends of the segment $[x_i , x_j]$ 
by the $\epsilon$-circle centered at $x_i$ or $x_j$; 
we denote by 
$S _{\epsilon} (x _i )$  a closed loop homotopic 
to the  $\epsilon$-circle centered at $x_i$
both of whose initial and terminal points  
are at the initial point of the segment $ [ x_i +\epsilon , x_j +\epsilon ] $,  
and 
$S _{\epsilon} (x_i )$ is defined in exactly the same way; 
we shall indicate by $ ^{\overline{\ \ }}$ 
an image through the natural projection $\Bbb{C} \longrightarrow  C$.   
Note that if $\epsilon$ is 
sufficiently small, this definition is independent of $\epsilon$.  
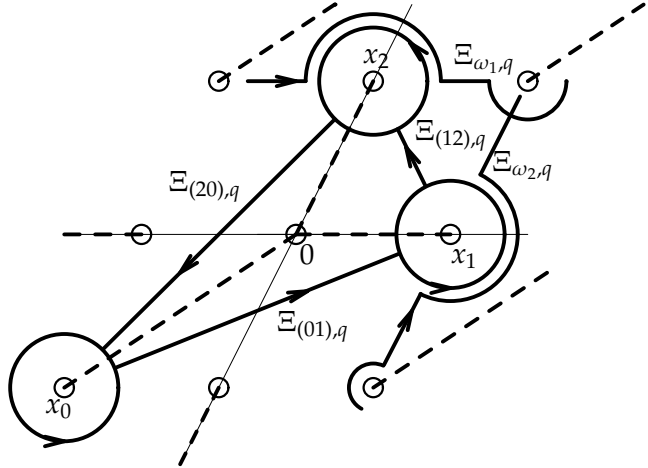
\begin{center}
\begin{minipage}{.5\textwidth} 
\unitlength 0.1in
\begin{picture}( 34.3500, 24.0000)(  4.8500,-28.0000)
%
\special{pn 4}%
\special{pa 920 1600}%
\special{pa 2320 1600}%
\special{fp}%
\put(21.4000,-16.5000){\makebox(0,0)[lt]{$0$}}%
%
\special{pn 4}%
\special{pa 2720 400}%
\special{pa 1520 2800}%
\special{fp}%
%
\special{pn 4}%
\special{pa 1920 1600}%
\special{pa 3320 1600}%
\special{fp}%
%
\special{pn 20}%
\special{pa 1120 2200}%
\special{pa 2320 1000}%
\special{fp}%
%
\special{pn 13}%
\special{ar 2520 800 50 50  0.0000000 6.2831853}%
%
\special{pn 13}%
\special{ar 920 2400 50 50  0.0000000 6.2831853}%
%
\special{pn 13}%
\special{ar 2120 1600 50 50  0.0000000 6.2831853}%
%
\special{pn 13}%
\special{ar 3320 800 50 50  0.0000000 6.2831853}%
%
\special{pn 13}%
\special{ar 2920 1600 50 50  0.0000000 6.2831853}%
%
\special{pn 13}%
\special{ar 1320 1600 50 50  0.0000000 6.2831853}%
%
\special{pn 13}%
\special{ar 1720 2400 50 50  0.0000000 6.2831853}%
%
\special{pn 13}%
\special{ar 2520 2400 50 50  0.0000000 6.2831853}%
%
\special{pn 13}%
\special{ar 1720 800 50 50  0.0000000 6.2831853}%
%
\special{pn 20}%
\special{ar 2520 800 280 280  0.0000000 6.2831853}%
%
\special{pn 20}%
\special{ar 920 2400 280 280  0.0000000 6.2831853}%
%
\special{pn 20}%
\special{ar 2920 1600 280 280  0.0000000 6.2831853}%
%
\special{pn 20}%
\special{pa 1520 1800}%
\special{pa 1608 1748}%
\special{fp}%
\special{pa 1520 1800}%
\special{pa 1574 1712}%
\special{fp}%
%
\special{pn 20}%
\special{pa 1186 2296}%
\special{pa 2656 1700}%
\special{fp}%
%
\special{pn 20}%
\special{pa 2790 1346}%
\special{pa 2646 1050}%
\special{fp}%
%
\special{pn 20}%
\special{pa 2180 1896}%
\special{pa 2078 1904}%
\special{fp}%
\special{pa 2180 1896}%
\special{pa 2094 1950}%
\special{fp}%
%
\special{pn 20}%
\special{pa 2690 1146}%
\special{pa 2712 1246}%
\special{fp}%
\special{pa 2690 1146}%
\special{pa 2758 1224}%
\special{fp}%
%
\special{pn 20}%
\special{pa 2710 600}%
\special{pa 2764 688}%
\special{fp}%
\special{pa 2710 600}%
\special{pa 2798 654}%
\special{fp}%
%
\special{pn 20}%
\special{pa 2920 1880}%
\special{pa 2820 1856}%
\special{fp}%
\special{pa 2920 1880}%
\special{pa 2820 1906}%
\special{fp}%
%
\special{pn 20}%
\special{pa 2120 1600}%
\special{pa 2520 800}%
\special{da 0.070}%
%
\special{pn 20}%
\special{pa 2120 1600}%
\special{pa 2920 1600}%
\special{da 0.070}%
%
\special{pn 20}%
\special{pa 2120 1600}%
\special{pa 920 2400}%
\special{da 0.070}%
%
\special{pn 20}%
\special{pa 1720 800}%
\special{pa 2320 400}%
\special{da 0.070}%
\special{pa 2320 400}%
\special{pa 2320 400}%
\special{da 0.070}%
%
\special{pn 20}%
\special{pa 3320 800}%
\special{pa 3920 400}%
\special{da 0.070}%
%
\special{pn 20}%
\special{pa 2520 2400}%
\special{pa 3420 1800}%
\special{da 0.070}%
%
\special{pn 20}%
\special{pa 1720 2400}%
\special{pa 1520 2800}%
\special{da 0.070}%
%
\special{pn 20}%
\special{pa 1320 1600}%
\special{pa 920 1600}%
\special{da 0.070}%
%
\special{pn 20}%
\special{ar 2520 800 350 350  3.1415927 6.2831853}%
%
\special{pn 20}%
\special{ar 2920 1600 348 348  5.1696672 6.2831853}%
\special{ar 2920 1600 348 348  0.0000000 2.0344439}%
%
\special{pn 20}%
\special{pa 3076 1286}%
\special{pa 3280 880}%
\special{fp}%
%
\special{pn 20}%
\special{ar 3320 800 200 200  6.2831853 6.2831853}%
\special{ar 3320 800 200 200  0.0000000 3.1415927}%
%
\special{pn 20}%
\special{pa 2870 800}%
\special{pa 3120 800}%
\special{fp}%
%
\special{pn 20}%
\special{pa 2170 800}%
\special{pa 1870 800}%
\special{fp}%
%
\special{pn 20}%
\special{ar 2520 2400 126 126  2.0736395 5.1421513}%
%
\special{pn 20}%
\special{pa 2766 1910}%
\special{pa 2576 2280}%
\special{fp}%
%
\special{pn 20}%
\special{pa 920 2680}%
\special{pa 820 2656}%
\special{fp}%
\special{pa 920 2680}%
\special{pa 820 2706}%
\special{fp}%
%
\special{pn 20}%
\special{pa 2120 800}%
\special{pa 2020 776}%
\special{fp}%
\special{pa 2120 800}%
\special{pa 2020 826}%
\special{fp}%
%
\special{pn 20}%
\special{pa 2720 2000}%
\special{pa 2654 2078}%
\special{fp}%
\special{pa 2720 2000}%
\special{pa 2698 2100}%
\special{fp}%
\put(20.2000,-20.0000){\makebox(0,0)[lt]{$\Xi _{(01) ,q}$}}%
\put(18.3500,-14.3000){\makebox(0,0)[rb]{$\Xi _{(20),q}$}}%
\put(27.3500,-11.6500){\makebox(0,0)[lb]{$\Xi _{(12) ,q}$}}%
\put(31.3500,-11.6500){\makebox(0,0)[lt]{$\Xi _{\omega _2 ,q}$}}%
\put(29.3500,-7.6500){\makebox(0,0)[lb]{$\Xi _{\omega _1 ,q}$}}%
\put(24.7000,-7.3000){\makebox(0,0)[lb]{$x_2$}}%
\put(29.2500,-16.7500){\makebox(0,0)[lt]{$x_1$}}%
\put(8.2500,-24.6000){\makebox(0,0)[lt]{$x_0$}}%
\end{picture}%
 \end{minipage} 
\end{center} 
\begin{thm}[generators of $H_1$] \label{thm_twisted_cycles} \ 
\begin{enumerate} 
\item The chains 
$\Xi _{(01) ,q}$, $\Xi _{(12) ,q}$, $\Xi _{(20) ,q}$, $\Xi _{\omega _1 ,q}$, 
$\Xi _{\omega _2 ,q}$ are cycles and generate 
$H _1 ( X_q , \mathcal{L} _{X_q} )$. 
\item We have the relation 
$\Xi _{(01) ,q} + \Xi _{(12) ,q} + \Xi _{(20) ,q} =0$ in 
$H _1 ( X_q , \mathcal{L} _{X_q} )$. 
\end{enumerate} 
\end{thm}   
\begin{proof} 
The remark {\ref{rem3.2}} tells us two facts. 
One is that  
the elements of $H _1 ( X_q , \mathcal{L} _{X_q} )$ 
are represented by 
linear combinations of 
 $l_0 \otimes \varsigma _{l_0}$, 
$l_1 \otimes \varsigma _{l_1}$, 
$l_2 \otimes \varsigma _{l_2}$, 
$l_{\omega _1} \otimes \varsigma _{l_{\omega _1}}$,  
$l_{\omega _2} \otimes \varsigma _{l_{\omega _2}}$. 
The other is that 
the topological euler number 
$\chi ( X_q )$ equals $4$. 
This implies that the dimension of 
$H _1 ( X_q , \mathcal{L} _{X_q} )$ equals $4$. 
We shall pick four linearly independent  elements.     
The chain $\frac{1}{c-1} l_i \otimes \varsigma _{l_i}
-\frac{1}{c-1} l_j \otimes \varsigma _{l_j}$ is a cycle homologous to 
$\Xi _{(ij) ,q}$.  
Immediately, we have the linear relation stated in this theorem. 
Then we obtain five elements of $H _1 ( X_q , \mathcal{L} _{X_q} )$ subject to a single linear relation.  
The intersection form $(Theorem  {\ref{intersection_form}})$ 
in the next section 
tells us that any four  of the five elements   
$\Xi _{(01) ,q}$, $\Xi _{(12) ,q}$, $\Xi _{(20) ,q}$, $\Xi _{\omega _1 ,q}$, 
$\Xi _{\omega _2 ,q}$ 
are linearly independent. 
\end{proof}

\section{The intersection form} \label{sec_intersection_form} 
The intersection form on twisted cycles is a bilinear form  between 
$H _1 \left( X _q , \mathcal{L} _{X_q} \right)$ and 
$H_1 \left( X_q , \mathcal{L} _{X_q} ^{\vee} \right)$, 
where $\mathcal{L} _{X_q} ^{\vee}$ is the local system dual to 
$\mathcal{L} _{X_q}$; 
the local sections of $\mathcal{L} _{X_q} ^{\vee}$ are generated 
by $g ^{-\alpha}$ instead of $g ^{\alpha}$.   
To describe the intersection form concretely, 
we shall take the generators 
$\Xi _{(01) ,q} ^{\vee}$, $\Xi _{(12) ,q} ^{\vee}$, 
$\Xi _{(20) ,q} ^{\vee}$, $\Xi _{\omega _1 ,q} ^{\vee}$ and  
$\Xi _{\omega _2 ,q} ^{\vee}$ 
of 
$H_1 \left( X_q , \mathcal{L} _{X_q} ^{\vee} \right)$ 
which are given by replacing $\varsigma$ (resp. $c$) by 
$\frac{1}{\varsigma}$ (resp. $\frac{1}{c}$) 
in the definitions of 
$\Xi _{(01) ,q}$, $\Xi _{(12) ,q}$, 
$\Xi _{(20) ,q}$, $\Xi _{\omega _1 ,q}$ and  
$\Xi _{\omega _2 ,q}$ 
respectively: 
\begin{align*} 
\varsigma \Longleftrightarrow \frac{1}{\varsigma} , \quad 
c \Longleftrightarrow \frac{1}{c} 
\quad ; \quad 
\Xi = \sum _{\gamma} a_{\gamma} ( c ) \gamma \otimes \varsigma _{\gamma} 
\Longleftrightarrow 
\Xi ^{\vee} = \sum _{\gamma} a_{\gamma} \left( \frac{1}{c} \right) 
\gamma \otimes \frac{1}{\varsigma _{\gamma}} , 
\end{align*} 
where $\gamma \otimes \varsigma _{\gamma}$ is a twisted singular 
$1$-simplex and $a _{\gamma} (X) \in \Bbb{C} (X)$. 
 
We define the ordering on the index set 
$J := \{ (01) , (12) , (20) , \omega _1 , \omega _2 \}$ 
of  the generators of $H_1$ :  
\begin{align*} 
 (01) \prec (12) \prec (20) \prec \omega _1 \prec \omega _2 .  
\end{align*} 
Now we have 
\begin{thm}[intersection form] \label{intersection_form}
We assume $\arg x_0 < \arg x_1 < \arg x_2$. 
Then the intersection form between 
$H _1 \left( X _q , \mathcal{L} _{X_q} \right)$ and 
$H_1 \left( X_q , \mathcal{L} _{X_q} ^{\vee} \right)$ 
is given by the matrix 
\begin{align*} 
\left[ \langle \Xi _{\mu ,q} , \Xi ^{\vee} _{\nu ,q} \rangle \right] 
_{\mu , \nu \in J} 
= 
\begin{bmatrix} 
-\frac{c+1}{c-1}  & \frac{1}{c-1} & \frac{c}{c-1} & 0 & 0 \\ 
\frac{c}{c-1} & -\frac{c+1}{c-1} & \frac{1}{c-1} & 0 & 0 \\ 
\frac{1}{c-1} & \frac{c}{c-1} &-\frac{c+1}{c-1}  & 0 & 0 \\ 
0 & 0 & 0 & 0 & 1 \\ 
0 & 0 & 0 & -1 & 0
\end{bmatrix}  . 
\end{align*} 
In particular, 
the (2,2)-cofactor of this matrix is equal to 
$\frac{c^2 + c +1}{(c-1)^2}$. 
\end{thm} 
\begin{proof} 
The intersection number 
$\langle \Xi _{\mu ,q} , \Xi _{\nu ,q} ^{\vee} \rangle$ 
can be calculated by summing up the local intersection numbers 
at each of the intersection points of the supports of 
$\Xi _{\mu ,q}$ and of $\Xi _{\nu ,q} ^{\vee}$. 
The local intersection number 
$\langle \gamma _1 \otimes \varsigma _{\gamma _1} , 
\gamma _2 \otimes \frac{1}{\varsigma _{\gamma _2}} \rangle _p$ 
at the intersection point $p$ of  their supports 
is the product of the ordinary local intersection number 
$\langle \gamma _1 , \gamma _2 \rangle _p$ and 
$\frac{\varsigma _{\gamma _1} (p)}{\varsigma _{\gamma _2} (p)}$. \\ 
\underline{In the case of $\mu = \nu = (i j ) 
\in \{ (01) , (12) , (20) \}$:} \\ 
We change the representative of the homology class 
$\Xi _{\nu ,q} ^{\vee}$ 
in such a way that its support meets that of $\Xi _{\mu ,q}$ 
at two points $p_i$, $p_j$, 
which belong to the supports 
of 
$\overline{S _{\epsilon} (x_i)} \otimes \varsigma _{\overline{S _{\epsilon} (x_i)} }$ and 
of 
$\overline{S _{\epsilon} (x_j)} \otimes \varsigma _{\overline{S _{\epsilon} (x_j)} }$ respectively: 
\begin{eqnarray*} 
 \langle \Xi _{\mu ,q} , \Xi _{\nu ,q} ^{\vee} \rangle 
&=& \frac{1}{c-1} \times (-1) \times 
\frac{\varsigma _{\overline{S _{\epsilon} (x_i ) } } (p _i )}{\varsigma 
_{\overline{[ x_i +\epsilon , x_j -\epsilon ]}} (p _i )} 
 - \frac{1}{c-1} \times 1 \times  
\frac{\varsigma _{\overline{S _{\epsilon} (x_j ) } } (p _j )}{\varsigma 
_{\overline{[ x_i +\epsilon , x_j -\epsilon ]}} (p _j )} \\ 
&=& \frac{1}{c-1} \times (-1) \times c -\frac{1}{c-1} \times 1 \times 1
= -\frac{c+1}{c-1} . 
\end{eqnarray*} 
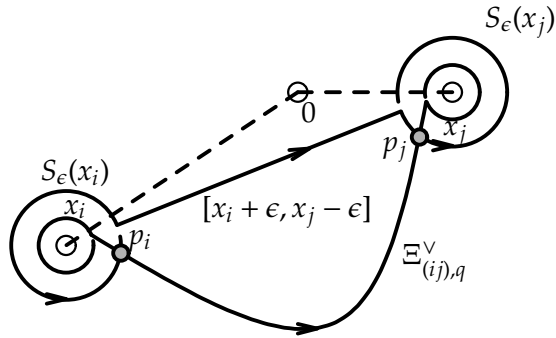
\begin{center}
\begin{minipage}{.5\textwidth} 
\unitlength 0.1in
\begin{picture}( 26.2800, 17.1500)(  1.7000,-18.9500)
\put(17.3500,-6.8500){\makebox(0,0)[lt]{$0$}}%
%
\special{pn 13}%
\special{ar 516 1436 50 50  0.0000000 6.2831853}%
%
\special{pn 13}%
\special{ar 1716 636 50 50  0.0000000 6.2831853}%
%
\special{pn 20}%
\special{pa 780 1330}%
\special{pa 2250 736}%
\special{fp}%
%
\special{pn 20}%
\special{pa 1776 930}%
\special{pa 1672 938}%
\special{fp}%
\special{pa 1776 930}%
\special{pa 1688 986}%
\special{fp}%
%
\special{pn 20}%
\special{pa 2516 916}%
\special{pa 2416 890}%
\special{fp}%
\special{pa 2516 916}%
\special{pa 2416 940}%
\special{fp}%
%
\special{pn 20}%
\special{pa 1716 636}%
\special{pa 2516 636}%
\special{da 0.070}%
%
\special{pn 20}%
\special{pa 1716 636}%
\special{pa 516 1436}%
\special{da 0.070}%
\put(12.2000,-11.8000){\makebox(0,0)[lt]{$[ x_i +\epsilon , x_j -\epsilon ]$}}%
\put(24.7000,-7.8500){\makebox(0,0)[lt]{$x_j$}}%
\put(6.2000,-12.9500){\makebox(0,0)[rb]{$x_i$}}%
%
\special{pn 13}%
\special{ar 2516 636 50 50  0.0000000 6.2831853}%
%
\special{pn 20}%
\special{pa 516 1716}%
\special{pa 416 1690}%
\special{fp}%
\special{pa 516 1716}%
\special{pa 416 1740}%
\special{fp}%
%
\special{pn 20}%
\special{ar 516 1436 284 284  6.1180366 6.2831853}%
\special{ar 516 1436 284 284  0.0000000 5.8934767}%
%
\special{pn 20}%
\special{ar 2516 636 284 284  2.9648838 6.2831853}%
\special{ar 2516 636 284 284  0.0000000 2.7672590}%
%
\special{pn 20}%
\special{ar 2516 636 142 142  2.9648838 6.2831853}%
\special{ar 2516 636 142 142  0.0000000 2.7672590}%
%
\special{pn 20}%
\special{ar 516 1436 142 142  6.1180366 6.2831853}%
\special{ar 516 1436 142 142  0.0000000 5.8904246}%
%
\special{pn 20}%
\special{pa 646 1380}%
\special{pa 680 1402}%
\special{pa 716 1424}%
\special{pa 750 1444}%
\special{pa 784 1466}%
\special{pa 820 1486}%
\special{pa 854 1508}%
\special{pa 888 1528}%
\special{pa 922 1548}%
\special{pa 958 1568}%
\special{pa 992 1588}%
\special{pa 1026 1608}%
\special{pa 1060 1626}%
\special{pa 1092 1644}%
\special{pa 1126 1662}%
\special{pa 1160 1680}%
\special{pa 1192 1698}%
\special{pa 1224 1714}%
\special{pa 1258 1730}%
\special{pa 1290 1746}%
\special{pa 1322 1760}%
\special{pa 1352 1774}%
\special{pa 1384 1786}%
\special{pa 1414 1798}%
\special{pa 1446 1810}%
\special{pa 1476 1822}%
\special{pa 1504 1830}%
\special{pa 1534 1840}%
\special{pa 1562 1848}%
\special{pa 1592 1854}%
\special{pa 1620 1860}%
\special{pa 1646 1866}%
\special{pa 1674 1868}%
\special{pa 1700 1872}%
\special{pa 1726 1872}%
\special{pa 1752 1872}%
\special{pa 1776 1872}%
\special{pa 1800 1870}%
\special{pa 1824 1866}%
\special{pa 1846 1860}%
\special{pa 1870 1854}%
\special{pa 1892 1846}%
\special{pa 1912 1838}%
\special{pa 1932 1826}%
\special{pa 1952 1814}%
\special{pa 1972 1800}%
\special{pa 1990 1786}%
\special{pa 2008 1768}%
\special{pa 2026 1752}%
\special{pa 2042 1732}%
\special{pa 2058 1712}%
\special{pa 2074 1690}%
\special{pa 2090 1668}%
\special{pa 2104 1644}%
\special{pa 2118 1620}%
\special{pa 2132 1594}%
\special{pa 2146 1566}%
\special{pa 2158 1538}%
\special{pa 2170 1510}%
\special{pa 2182 1480}%
\special{pa 2194 1450}%
\special{pa 2206 1418}%
\special{pa 2216 1386}%
\special{pa 2226 1354}%
\special{pa 2236 1320}%
\special{pa 2246 1286}%
\special{pa 2256 1252}%
\special{pa 2266 1216}%
\special{pa 2276 1180}%
\special{pa 2284 1144}%
\special{pa 2294 1106}%
\special{pa 2302 1070}%
\special{pa 2310 1032}%
\special{pa 2318 994}%
\special{pa 2326 956}%
\special{pa 2334 918}%
\special{pa 2342 878}%
\special{pa 2350 840}%
\special{pa 2358 800}%
\special{pa 2366 762}%
\special{pa 2374 722}%
\special{pa 2380 686}%
\special{sp}%
%
\special{pn 20}%
\special{pa 1816 1870}%
\special{pa 1716 1846}%
\special{fp}%
\special{pa 1816 1870}%
\special{pa 1716 1896}%
\special{fp}%
%
\special{pn 20}%
\special{sh 0.300}%
\special{ar 800 1476 40 40  0.0000000 6.2831853}%
%
\special{pn 20}%
\special{sh 0.300}%
\special{ar 2346 870 40 40  0.0000000 6.2831853}%
\put(8.4500,-14.7000){\makebox(0,0)[lb]{$p_i$}}%
\put(22.8500,-8.7000){\makebox(0,0)[rt]{$p_j$}}%
\put(3.8500,-11.2500){\makebox(0,0)[lb]{$S _{\epsilon} (x_i)$}}%
\put(26.8500,-3.5000){\makebox(0,0)[lb]{$S_{\epsilon} (x_j)$}}%
\put(22.5500,-14.1500){\makebox(0,0)[lt]{$\Xi _{(i j) ,q} ^{\vee}$}}%
\end{picture}%
 \end{minipage} 
\end{center}  
\underline{In the case of $\mu \not= \nu$, 
$\mu , \nu \in \{ (01) , (12) , (20) \}$:} \\ 
We change the representative of the homology class 
$\Xi _{(jk) ,q} ^{\vee}$ 
in such a way that its support meets that of $\Xi _{(ij) ,q}$ 
at a point $p_j$ on the support of 
$\overline{S _{\epsilon} (x_j)} \otimes \varsigma _{\overline{S _{\epsilon} (x_j)} }$: 
\begin{eqnarray*} 
\langle \Xi _{(01) ,q} , \Xi _{(12) ,q} ^{\vee} \rangle 
= \langle \Xi _{(12) ,q} , \Xi _{(20) ,q} ^{\vee} \rangle 
= \langle \Xi _{(20) ,q} , \Xi _{(01) ,q} ^{\vee} \rangle 
=& -\frac{1}{c-1} \times (-1) \times 
\frac{\varsigma _{\overline{S _{\epsilon} (x_j ) } } (p _j )}{\varsigma _{\overline{[ x_i +\epsilon , x_j -\epsilon ]}} (p _j )} 
&= \frac{1}{c-1} , \\ 
\langle \Xi _{(12) ,q} , \Xi _{(01) ,q} ^{\vee} \rangle 
= \langle \Xi _{(20) ,q} , \Xi _{(12) ,q} ^{\vee} \rangle 
= \langle \Xi _{(01) ,q} , \Xi _{(20) ,q} ^{\vee} \rangle 
=& \frac{1}{c-1} \times 1 \times 
\frac{\varsigma _{\overline{S _{\epsilon} (x_j ) } } (p _j )}{\varsigma _{\overline{[ x_i +\epsilon , x_j -\epsilon ]}} (p _j )} 
&= \frac{c}{c-1} . 
\end{eqnarray*} 
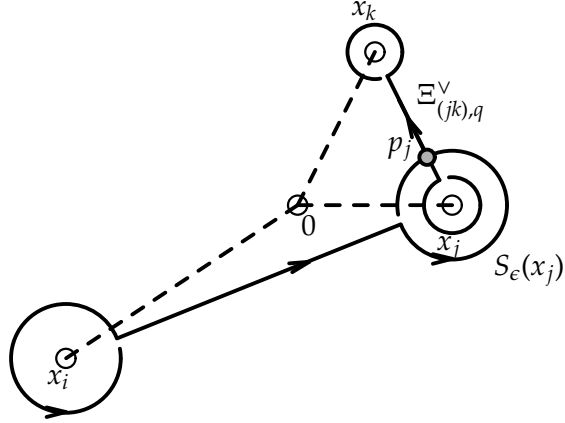
\begin{center} 
\begin{minipage}{.5\textwidth} 
\unitlength 0.1in
\begin{picture}( 25.6600, 22.4000)(  6.3400,-28.5500)
\put(21.3700,-18.0000){\makebox(0,0)[lt]{$0$}}%
%
\special{pn 13}%
\special{ar 2518 950 50 50  0.0000000 6.2831853}%
%
\special{pn 13}%
\special{ar 918 2550 50 50  0.0000000 6.2831853}%
%
\special{pn 13}%
\special{ar 2118 1750 50 50  0.0000000 6.2831853}%
%
\special{pn 13}%
\special{ar 2918 1750 50 50  0.0000000 6.2831853}%
%
\special{pn 20}%
\special{pa 1182 2446}%
\special{pa 2652 1850}%
\special{fp}%
%
\special{pn 20}%
\special{pa 2788 1496}%
\special{pa 2642 1200}%
\special{fp}%
%
\special{pn 20}%
\special{pa 2178 2046}%
\special{pa 2074 2054}%
\special{fp}%
\special{pa 2178 2046}%
\special{pa 2090 2100}%
\special{fp}%
%
\special{pn 20}%
\special{pa 2688 1296}%
\special{pa 2710 1396}%
\special{fp}%
\special{pa 2688 1296}%
\special{pa 2754 1374}%
\special{fp}%
%
\special{pn 20}%
\special{pa 2918 2030}%
\special{pa 2818 2006}%
\special{fp}%
\special{pa 2918 2030}%
\special{pa 2818 2056}%
\special{fp}%
%
\special{pn 20}%
\special{pa 2118 1750}%
\special{pa 2518 950}%
\special{da 0.070}%
%
\special{pn 20}%
\special{pa 2118 1750}%
\special{pa 2918 1750}%
\special{da 0.070}%
%
\special{pn 20}%
\special{pa 2118 1750}%
\special{pa 918 2550}%
\special{da 0.070}%
%
\special{pn 20}%
\special{pa 918 2830}%
\special{pa 818 2806}%
\special{fp}%
\special{pa 918 2830}%
\special{pa 818 2856}%
\special{fp}%
\put(27.3200,-13.1500){\makebox(0,0)[lb]{$\Xi _{(jk) ,q} ^{\vee}$}}%
\put(23.9700,-7.8500){\makebox(0,0)[lb]{$x_k$}}%
\put(28.4000,-19.1000){\makebox(0,0)[lt]{$x_j$}}%
\put(8.2200,-26.1000){\makebox(0,0)[lt]{$x_i$}}%
%
\special{pn 20}%
\special{ar 918 2550 284 284  6.1033318 6.2831853}%
\special{ar 918 2550 284 284  0.0000000 5.8934767}%
%
\special{pn 20}%
\special{ar 2918 1750 284 284  2.9648838 6.2831853}%
\special{ar 2918 1750 284 284  0.0000000 2.7672590}%
%
\special{pn 20}%
\special{ar 2518 950 142 142  1.3734008 6.2831853}%
\special{ar 2518 950 142 142  0.0000000 1.0912770}%
%
\special{pn 20}%
\special{ar 2918 1750 146 146  4.5975124 6.2831853}%
\special{ar 2918 1750 146 146  0.0000000 4.2487414}%
%
\special{pn 20}%
\special{pa 2852 1620}%
\special{pa 2582 1076}%
\special{fp}%
%
\special{pn 20}%
\special{sh 0.300}%
\special{ar 2788 1500 40 40  0.0000000 6.2831853}%
\put(31.3200,-20.0500){\makebox(0,0)[lt]{$S _{\epsilon} (x_j)$}}%
\put(27.1200,-15.2000){\makebox(0,0)[rb]{$p_j$}}%
\end{picture}%
 \end{minipage} 
\end{center} 
\underline{In the case of $\mu , \nu \in \{\omega _1 , \omega _2 \}$:} \\ 
The ordinary intersection form 
$\langle \bullet , \bullet \rangle 
: H_1 ( X_q , \Bbb{C} ) \times H_1 ( X_q , \Bbb{C} ) 
\longrightarrow \Bbb{C}$ 
is given by 
$\left[ \langle l _{\omega _i } , l _{\omega _j } \rangle \right] _{i,j} = 
\begin{bmatrix} 0 & 1 \\ -1 & 0 \end{bmatrix}$, 
and 
$\frac{\varsigma _{l_{\omega _i}}(p)}{\varsigma _{l_{\omega _j}} (p)} 
=1$ at the intersection point $p$ of
 the supports of $l_{\omega _i}$ and of $l_{\omega _j}$. 
Hence, 
\begin{align*} 
\left[ \langle \Xi _{\omega _i ,q} , \Xi _{\omega _j ,q } ^{\vee} \rangle \right] _{ij}
= \begin{bmatrix} 0 & 1 \\ -1 & 0 \end{bmatrix} . 
\end{align*}  
\underline{In other cases:} \\ 
We can change the representative of the homology class 
$\Xi _{\nu ,q} ^{\vee}$ in such a way that 
its support  does not meet that of $\Xi _{\mu ,q}$. 
\end{proof}


\section{The connection matrix} \label{sec_connection_matrix} 
Put 
$\overline{S} := \left\{ \left. 
q= (\overline{x_0} , \overline{x_1} , \overline{x_2} ) 
\ \right| \ x_0 + x_1 + x_2 = 0 \right\} 
\subset C \times C \times C$, 
$D^{ij} := \left\{ \left. q \in \overline{S} \ \right| \ x_i -x_j \in \Gamma \right\} $ and 
$ D ^i _{\infty} := \left\{ \left. q \in \overline{S} \ \right| \ x_i \in \Gamma  
\right\}$. 
We denote the {\it configuration space} by 
$S := \overline{S} \setminus \left( \bigcup _{i \not= j} D ^{ij} \cup 
\bigcup _i D ^i _{\infty} \right)$ 
and the point $\left( \overline{\frac{\omega _i}{2}} , 
\overline{\frac{\omega _j}{2} } , \overline{\frac{\omega _k}{2}}  \right)$ 
on $S$ by $q_{(ijk)}$, 
where $\omega _0 := - \left( \omega _1 + \omega _2 \right)$. 
Let $\mathcal{S}$ be the sheaf over $S$ of germs of functions defined by 
$\Bbb{C}$-linear combinations of $\left\{ F _{\mu} \right\} _{\mu}$; 
the stalk over $q_{(ijk)}$ is the vector space spanned by germs of 
$ \left\{ \int _{\Xi _{\mu , q  } } g^{\alpha} dt \right\} _{\mu \in \left\{ (01) , (20) , \omega _1 , \omega _2 \right\} }$ 
at $q=q_{(ijk)}$. 

In this chapter, we shall describe the linear isomorphisms among stalks of 
$\mathcal{S}$ over $q_{(012)}$, $q_{(210)}$ and $q_{(102)}$, 
of which the corresponding matrices  are called the {\it connection matrices}, 
induced by the analytic continuations of 
$F _{\mu} (q) := \int _{\Xi _{\mu ,q}} g^{\alpha} dt $ 
along four paths (given later) on $S$. 
These linear isomorphisms are identified with the isomorphisms 
among the homology groups 
$\left\{ H_1 (X_q , \mathcal{L} _{X_q} ) \right\} 
_{q \in \left\{ q_{(012)} , q_{(210)} , q_{(102)} \right\} }$,   
which are induced by pulling back the fibration $\bigcup X_q \longrightarrow S$ 
along the paths on $S$. 
Put $\mathcal{H} (q) := 
H_1 (X_q , \mathcal{L} _{X_q} ) $ 
for brevity. 

We consider  the paths in $S$, with $\overline{x_1}$ fixed, 
whose initial (resp. terminal) point is $q_{(012)}$ (resp. $q_{(210)}$). 
By their definitions, 
these paths are on the subset 
$\left\{ \left. q = \left( \overline{x_0} , \overline{x_1} ,\overline{x_2} \right) 
\in \overline{S} \ \right| \ x_1 = \frac{\omega _1}{2} \right\}$, which meets the singular loci at six points ; 
it meets $D^{02}$ at four points 
$q _{(02)} ^{1,0}$, 
 $q _{(02)} ^{-1,0}$, 
$q _{(02)} ^{1,2}$, 
$q _{(02)} ^{-1,2}$,  
meets both $D^{01}$ and $D^{2} _{\infty}$ at the same point $q _{(02)} ^{0,2}$, 
and 
 meets both $D^{12}$ and $D^{0} _{\infty}$ at the same point $q _{(02)} ^{2,2}$, 
where 
\begin{align*} 
q _{(02)} ^{m_1 ,m_2} 
:= 
\left( \overline{\frac{\omega _0}{2} +
\frac{m_1 \omega _1 + m_2 \omega _2}{4} } , 
\overline{\frac{\omega _1}{2} } , 
\overline{\frac{\omega _2}{2} -
\frac{m_1 \omega _1 + m_2 \omega _2}{4} } 
\right) . 
\end{align*} 
We define four paths corresponding to the intersection points 
$q _{(02)} ^{1,0}$, 
 $q _{(02)} ^{-1,0}$, 
$q _{(02)} ^{1,2}$, 
$q _{(02)} ^{-1,2}$ 
with $D^{02}$ : 
The path  $\gamma _{(02)} ^{m_1 , m_2}$ 
is defined by slightly deforming a path from 
$q_{(012)}$ to $q_{(210)}$ via $q_{(02)} ^{m_1 , m_2}$ 
in such a way that $\gamma _{(02)} ^{m_1 , m_2}$ avoids $q _{(02)} ^{m_1 , m_2}$ : 
\begin{align*} 
\gamma _{(02)} ^{m_1 ,m_2} (s)  
:= 
\left( \overline{\frac{\omega _0}{2} +
\frac{m_1 \omega _1 + m_2 \omega _2}{2} s } , 
\overline{\frac{\omega _1}{2} } , 
\overline{\frac{\omega _2}{2} -
\frac{m_1 \omega _1 + m_2 \omega _2}{2} s } 
\right) , \quad 
0 \leq s \leq \frac{1}{2} - \varepsilon , \  
\frac{1}{2} +\varepsilon \leq s \leq 1 , 
\end{align*} 
where $\varepsilon$ is a small positive number. 
In a similar fashion, 
we define  the paths $\gamma _{(01)} ^{0 ,1 }$, 
$\gamma _{(01)} ^{0 ,-1}$, 
$\gamma _{(01)} ^{2 ,1}$, 
$\gamma _{(01)} ^{2 ,-1}$, 
with $\overline{x_2}$ fixed, 
whose initial (resp. terminal) point is $q _{(012)}$ (resp. $q_{(102)}$) : 
\begin{align*} 
\gamma _{(01)} ^{m_1 ,m_2} (s)  
:= 
\left( \overline{\frac{\omega _0}{2} +
\frac{m_1 \omega _1 + m_2 \omega _2}{2} s } ,  
\overline{\frac{\omega _1}{2} -
\frac{m_1 \omega _1 + m_2 \omega _2}{2} s } , 
\overline{\frac{\omega _2}{2} } 
\right) , \quad 
0 \leq s \leq \frac{1}{2} - \varepsilon , \  
\frac{1}{2} +\varepsilon \leq s \leq 1 .  
\end{align*}  
Note that the path $\gamma _{(ij)} ^{m_1 ,m_2}$ 
can be realized as 
a composition of 
$\gamma _{(02)} ^{1 ,0}$, 
$\gamma _{(02)} ^{-1 ,0}$, 
$\gamma _{(01)} ^{0 ,1}$ 
and 
$\gamma _{(01)} ^{0 ,-1}$, 
so is a path,     
with $\overline{x_0}$ fixed, 
whose initial (resp. terminal) point is $q _{(012)}$ (resp. $q_{(021)}$).  
For example, 
$\gamma _{(02)} ^{\pm 1 ,2} 
= ( \gamma _{(01)} ^{0 ,1} ) ^{\mp 1} \circ ( \gamma _{(01)} ^{0 ,-1} ) ^{\mp 1} 
\circ ( \gamma _{(02)} ^{\pm 1 ,0} ) ^{\pm 1} \circ 
 ( \gamma _{(01)} ^{0 ,-1} ) ^{\pm 1} \circ ( \gamma _{(01)} ^{0 ,1} ) ^{\pm 1} 
$. 
(The following figure depicts the behavior of the representative 
$x_{0,(ij)} ^{m_1 ,m_2} (s)$ of the $\overline{x_0}$-component of 
$\gamma _{(ij)} ^{m_1 ,m_2} (s)$, 
where $q _{0 ,(ij)} ^{m_1 ,m_2}$ denotes 
the representative of the $\overline{x_0}$-component of 
$q _{(ij)} ^{m_1 ,m_2} $.) \\ 
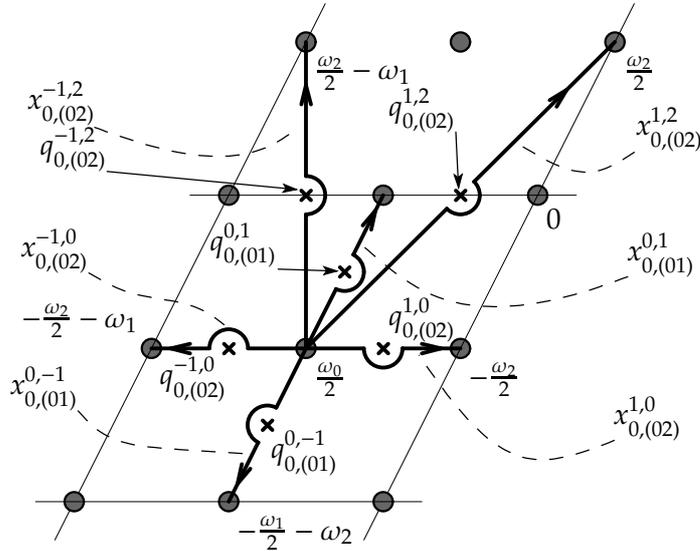
\begin{minipage}{.5\textwidth} 
\input{path.tex}
\end{minipage} 
\\ 

For brevity, we denote by $\tau _{(ij)}$ 
the transposition (belonging to the symmetric group $\frak{S} _3$) ; 
$\tau _{(ij)} (i) =j$, 
$\tau _{(ij)} (j) =i$ and 
$\tau _{(ij)} (k) =k$ for $k \not= i ,j$. 

We calculate the linear isomorphism 
$\left( \gamma _{(ij)} ^{m_1 ,m_2} \right) _{\ast}$ 
between 
$\mathcal{H}  \left( {q_{(012)}} \right)$ 
and 
$\mathcal{H} \left( {q_{(\tau _{(ij)}(0) , \tau _{(ij)} (1) , \tau _{(ij)} (2)  )}} \right)$ 
by the twisted Picard-Lefschetz formula ({\cite{givental}}) : 
The difference between a twisted cycle $\Xi$ and the transformed one 
$\left( \gamma _{(ij)} ^{m_1 ,m_2} \right) _{\ast} (\Xi )$  
is equal to the product of 
 the {\it vanishing cycle}, 
which vanishes at the singular point 
 $q _{(ij)} ^{m_1 ,m_2}$,  
and the constant factor which is calculated from the intersection numbers 
with the vanishing cycle and its {\it local monodromy} transformation. 

Let $\Delta _{(ij)} ^{m_1 ,m_2} (q)$ be the cycle vanishing at the singular point 
$q _{(ij)} ^{m_1 ,m_2}$, 
as $q$ tends to $q_{(ij)} ^{m_1 ,m_2}$ 
along an extended path of 
$\gamma _{(ij)} ^{m_1 ,m_2}$ that touches $q_{(ij)} ^{m_1 ,m_2}$. 
We have the following : 
\begin{prop}[vanishing cycles]  \label{vanishing_cycle}
\begin{align*} 
\Delta _{(02)} ^{-1 ,0} ( q _{(012)} ) 
&= 
-\Xi _{(20) , q_{(012)}} -\Xi _{\omega _1 , q_{(012)}} -\Xi _{\omega _2 , q_{(012)}},
\\ 
\Delta _{(02)} ^{1 ,0} ( q _{(012)} ) 
&= 
-\Xi _{(20) , q_{(012)}} -\Xi _{\omega _2 , q_{(012)}}, \\ 
\Delta _{(01)} ^{0 ,-1} ( q _{(012)} ) 
&= 
\Xi _{(01) , q_{(012)}} -\Xi _{\omega _1 , q_{(012)}} -\Xi _{\omega _2 , q_{(012)}},
\\ 
\Delta _{(01)} ^{0 ,1} ( q _{(012)} ) 
&= 
\Xi _{(01) , q_{(012)}} -\Xi _{\omega _1 , q_{(012)}} .  
\end{align*} 
\end{prop} 
\begin{proof} 
The vanishing cycle $\Delta _{(ij)} ^{m_1 ,m_2} ( q _{(012)})$ 
is a path on $X_{q_{(012)}}$ 
connecting two points $\overline{\frac{\omega _i}{2}}$, 
$\overline{\frac{\omega _j}{2}}$ depending on   
the path $\gamma _{(ij)} ^{m_1 , m_2}$.  
Replacing it by the linear combination of 
$\Xi _{(01) , q_{(012)}}$, $\Xi _{(20) , q_{(012)}}$, $\Xi _{\omega _1 , q_{(012)}}$ 
 and $\Xi _{\omega _2 , q_{(012)}}$, 
we obtain the assertion. 
(See the following figures,  
in which the dotted lines indicate the vanishing cycles.) \\  
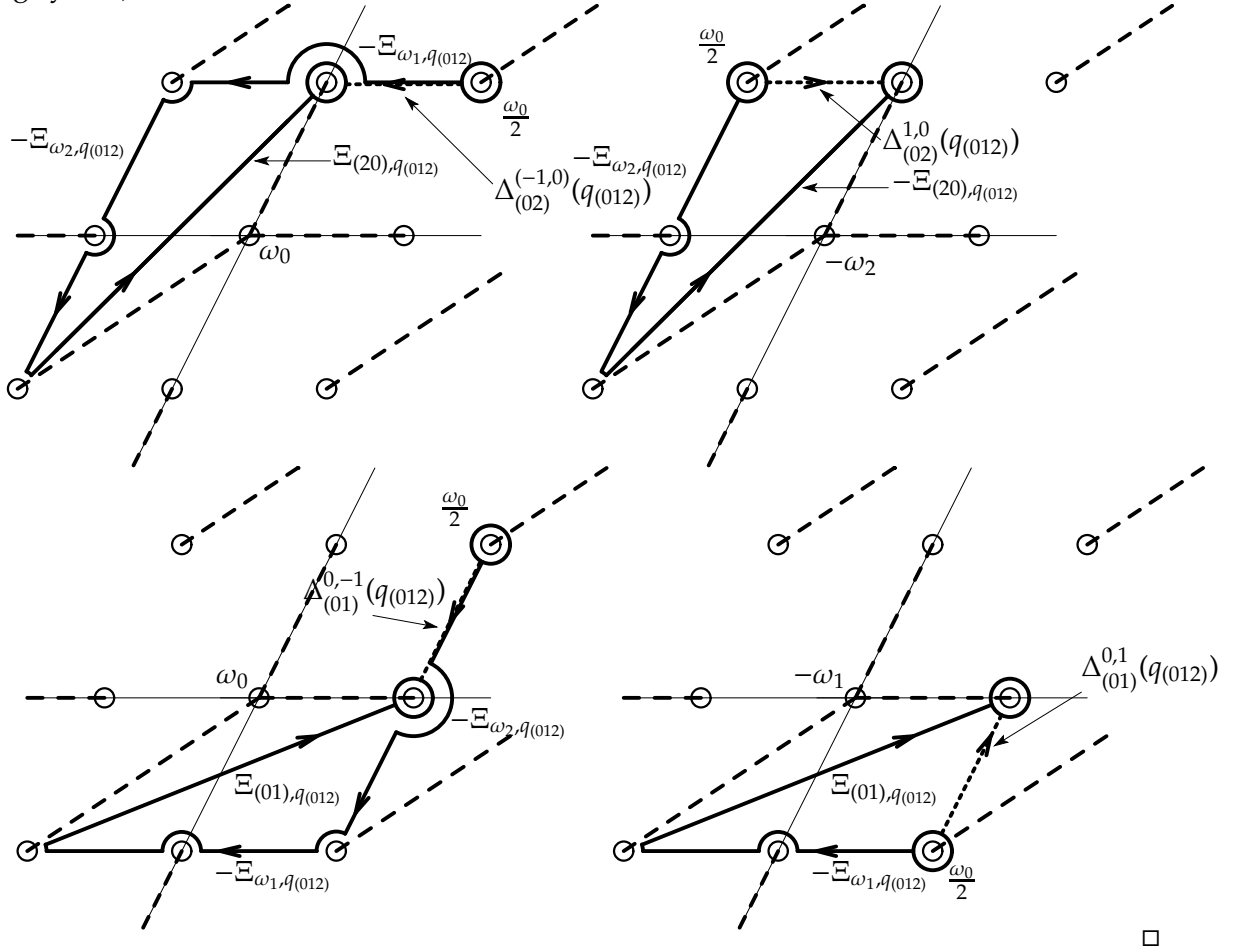
\begin{minipage}{.5\textwidth} 
\unitlength 0.1in
\begin{picture}( 20.3500, 24.0000)(0,-24.0000) 
%
\special{pn 4}%
\special{pa 600 1200}%
\special{pa 2000 1200}%
\special{fp}%
%
\special{pn 4}%
\special{pa 2400 0}%
\special{pa 1200 2400}%
\special{fp}%
%
\special{pn 4}%
\special{pa 1600 1200}%
\special{pa 3000 1200}%
\special{fp}%
%
\special{pn 20}%
\special{pa 800 1800}%
\special{pa 2000 600}%
\special{fp}%
%
\special{pn 13}%
\special{ar 2200 400 50 50  0.0000000 6.2831853}%
%
\special{pn 13}%
\special{ar 600 2000 50 50  0.0000000 6.2831853}%
%
\special{pn 13}%
\special{ar 1800 1200 50 50  0.0000000 6.2831853}%
%
\special{pn 13}%
\special{ar 3000 400 50 50  0.0000000 6.2831853}%
%
\special{pn 13}%
\special{ar 2600 1200 50 50  0.0000000 6.2831853}%
%
\special{pn 13}%
\special{ar 1000 1200 50 50  0.0000000 6.2831853}%
%
\special{pn 13}%
\special{ar 1400 2000 50 50  0.0000000 6.2831853}%
%
\special{pn 13}%
\special{ar 2200 2000 50 50  0.0000000 6.2831853}%
%
\special{pn 13}%
\special{ar 1400 400 50 50  0.0000000 6.2831853}%
%
\special{pn 20}%
\special{pa 1200 1400}%
\special{pa 1112 1454}%
\special{fp}%
\special{pa 1200 1400}%
\special{pa 1148 1488}%
\special{fp}%
%
\special{pn 20}%
\special{pa 1800 1200}%
\special{pa 2200 400}%
\special{da 0.070}%
%
\special{pn 20}%
\special{pa 1800 1200}%
\special{pa 2600 1200}%
\special{da 0.070}%
%
\special{pn 20}%
\special{pa 1800 1200}%
\special{pa 600 2000}%
\special{da 0.070}%
%
\special{pn 20}%
\special{pa 1400 400}%
\special{pa 2000 0}%
\special{da 0.070}%
\special{pa 2000 0}%
\special{pa 2000 0}%
\special{da 0.070}%
%
\special{pn 20}%
\special{pa 3000 400}%
\special{pa 3600 0}%
\special{da 0.070}%
%
\special{pn 20}%
\special{pa 2200 2000}%
\special{pa 3100 1400}%
\special{da 0.070}%
%
\special{pn 20}%
\special{pa 1400 2000}%
\special{pa 1200 2400}%
\special{da 0.070}%
%
\special{pn 20}%
\special{pa 1000 1200}%
\special{pa 600 1200}%
\special{da 0.070}%
%
\special{pn 20}%
\special{pa 1700 400}%
\special{pa 1800 426}%
\special{fp}%
\special{pa 1700 400}%
\special{pa 1800 376}%
\special{fp}%
%
\special{pn 20}%
\special{pa 800 1600}%
\special{pa 868 1522}%
\special{fp}%
\special{pa 800 1600}%
\special{pa 822 1500}%
\special{fp}%
%
\special{pn 20}%
\special{ar 2200 400 100 100  0.0000000 6.2831853}%
%
\special{pn 20}%
\special{ar 3000 400 100 100  0.0000000 6.2831853}%
%
\special{pn 20}%
\special{pa 2900 400}%
\special{pa 2400 400}%
\special{fp}%
%
\special{pn 20}%
\special{pa 2000 400}%
\special{pa 1500 400}%
\special{fp}%
%
\special{pn 20}%
\special{ar 1400 400 102 102  0.0996687 2.0344439}%
%
\special{pn 20}%
\special{ar 1000 1200 100 100  5.1760366 6.2831853}%
\special{ar 1000 1200 100 100  0.0000000 2.0344439}%
%
\special{pn 20}%
\special{ar 600 2000 102 102  5.1760366 5.4977871}%
%
\special{pn 20}%
\special{pa 1356 490}%
\special{pa 1046 1110}%
\special{fp}%
%
\special{pn 20}%
\special{pa 956 1290}%
\special{pa 646 1910}%
\special{fp}%
%
\special{pn 20}%
\special{pa 670 1930}%
\special{pa 2130 470}%
\special{fp}%
%
\special{pn 20}%
\special{ar 2200 400 200 200  3.1415927 6.2831853}%
%
\special{pn 20}%
\special{pa 2300 410}%
\special{pa 2900 410}%
\special{dt 0.054}%
%
\special{pn 20}%
\special{pa 2500 410}%
\special{pa 2600 436}%
\special{fp}%
\special{pa 2500 410}%
\special{pa 2600 386}%
\special{fp}%
\put(31.0000,-5.0000){\makebox(0,0)[lt]{$\frac{\omega _0}{2}$}}%
\put(18.5000,-12.5000){\makebox(0,0)[lt]{$\omega _0$}}%
\put(30.7000,-8.5000){\makebox(0,0)[lt]{$\Delta _{(02)} ^{(-1 ,0)} (q_{(012)}) $}}%
\put(23.6000,-3.2000){\makebox(0,0)[lb]{$-\Xi _{\omega _1 ,q_{(012)}}$}}%
\put(11.6500,-8.0500){\makebox(0,0)[rb]{$-\Xi _{\omega _2 ,q_{(012)}}$}}%
\put(22.3000,-7.3000){\makebox(0,0)[lt]{$\Xi _{(20) ,q_{(012)}}$}}%
%
\special{pn 4}%
\special{pa 2220 810}%
\special{pa 1830 810}%
\special{fp}%
\special{sh 1}%
\special{pa 1830 810}%
\special{pa 1898 830}%
\special{pa 1884 810}%
\special{pa 1898 790}%
\special{pa 1830 810}%
\special{fp}%
%
\special{pn 4}%
\special{pa 3060 930}%
\special{pa 2610 480}%
\special{fp}%
\special{sh 1}%
\special{pa 2610 480}%
\special{pa 2644 542}%
\special{pa 2648 518}%
\special{pa 2672 514}%
\special{pa 2610 480}%
\special{fp}%
\end{picture}%
 \end{minipage} 
\hspace{-.1\textwidth}
\begin{minipage}{.5\textwidth} 
\unitlength 0.1in
\begin{picture}(25 ,24)(0,-28) 
%
\special{pn 4}%
\special{pa 600 1600}%
\special{pa 2000 1600}%
\special{fp}%
%
\special{pn 4}%
\special{pa 2400 400}%
\special{pa 1200 2800}%
\special{fp}%
%
\special{pn 4}%
\special{pa 1600 1600}%
\special{pa 3000 1600}%
\special{fp}%
%
\special{pn 20}%
\special{pa 800 2200}%
\special{pa 2000 1000}%
\special{fp}%
%
\special{pn 13}%
\special{ar 2200 800 50 50  0.0000000 6.2831853}%
%
\special{pn 13}%
\special{ar 600 2400 50 50  0.0000000 6.2831853}%
%
\special{pn 13}%
\special{ar 1800 1600 50 50  0.0000000 6.2831853}%
%
\special{pn 13}%
\special{ar 3000 800 50 50  0.0000000 6.2831853}%
%
\special{pn 13}%
\special{ar 2600 1600 50 50  0.0000000 6.2831853}%
%
\special{pn 13}%
\special{ar 1000 1600 50 50  0.0000000 6.2831853}%
%
\special{pn 13}%
\special{ar 1400 2400 50 50  0.0000000 6.2831853}%
%
\special{pn 13}%
\special{ar 2200 2400 50 50  0.0000000 6.2831853}%
%
\special{pn 13}%
\special{ar 1400 800 50 50  0.0000000 6.2831853}%
%
\special{pn 20}%
\special{pa 1200 1800}%
\special{pa 1112 1854}%
\special{fp}%
\special{pa 1200 1800}%
\special{pa 1148 1888}%
\special{fp}%
%
\special{pn 20}%
\special{pa 1800 1600}%
\special{pa 2200 800}%
\special{da 0.070}%
%
\special{pn 20}%
\special{pa 1800 1600}%
\special{pa 2600 1600}%
\special{da 0.070}%
%
\special{pn 20}%
\special{pa 1800 1600}%
\special{pa 600 2400}%
\special{da 0.070}%
%
\special{pn 20}%
\special{pa 1400 800}%
\special{pa 2000 400}%
\special{da 0.070}%
\special{pa 2000 400}%
\special{pa 2000 400}%
\special{da 0.070}%
%
\special{pn 20}%
\special{pa 3000 800}%
\special{pa 3600 400}%
\special{da 0.070}%
%
\special{pn 20}%
\special{pa 2200 2400}%
\special{pa 3100 1800}%
\special{da 0.070}%
%
\special{pn 20}%
\special{pa 1400 2400}%
\special{pa 1200 2800}%
\special{da 0.070}%
%
\special{pn 20}%
\special{pa 1000 1600}%
\special{pa 600 1600}%
\special{da 0.070}%
%
\special{pn 20}%
\special{pa 1800 800}%
\special{pa 1700 776}%
\special{fp}%
\special{pa 1800 800}%
\special{pa 1700 826}%
\special{fp}%
%
\special{pn 20}%
\special{pa 796 2000}%
\special{pa 864 1924}%
\special{fp}%
\special{pa 796 2000}%
\special{pa 820 1900}%
\special{fp}%
%
\special{pn 20}%
\special{ar 2200 800 100 100  0.0000000 6.2831853}%
%
\special{pn 20}%
\special{ar 1400 800 100 100  0.0000000 6.2831853}%
%
\special{pn 20}%
\special{pa 1356 890}%
\special{pa 1046 1510}%
\special{fp}%
%
\special{pn 20}%
\special{pa 646 2306}%
\special{pa 956 1690}%
\special{fp}%
%
\special{pn 20}%
\special{ar 1000 1600 100 100  5.1760366 6.2831853}%
\special{ar 1000 1600 100 100  0.0000000 2.0344439}%
%
\special{pn 20}%
\special{ar 600 2400 100 100  5.1760366 5.4977871}%
%
\special{pn 20}%
\special{pa 670 2330}%
\special{pa 2130 870}%
\special{fp}%
%
\special{pn 20}%
\special{pa 1500 800}%
\special{pa 2100 800}%
\special{dt 0.054}%
\put(13.0000,-7.0000){\makebox(0,0)[rb]{$\frac{\omega _0}{2}$}}%
\put(18.0000,-17.0000){\makebox(0,0)[lt]{$-\omega _2$}}%
\put(11.0000,-13.0000){\makebox(0,0)[rb]{$-\Xi _{\omega _2 ,q_{(012)}}$}}%
\put(21.0000,-10.0000){\makebox(0,0)[lt]{$\Delta _{(02)} ^{1,0} (q_{(012)})$}}%
\put(21.5000,-12.6000){\makebox(0,0)[lt]{$-\Xi _{(20) ,q_{(012)}}$}}%
%
\special{pn 4}%
\special{pa 2090 1100}%
\special{pa 1780 840}%
\special{fp}%
\special{sh 1}%
\special{pa 1780 840}%
\special{pa 1818 898}%
\special{pa 1822 874}%
\special{pa 1844 868}%
\special{pa 1780 840}%
\special{fp}%
%
\special{pn 4}%
\special{pa 2130 1350}%
\special{pa 1690 1350}%
\special{fp}%
\special{sh 1}%
\special{pa 1690 1350}%
\special{pa 1758 1370}%
\special{pa 1744 1350}%
\special{pa 1758 1330}%
\special{pa 1690 1350}%
\special{fp}%
\end{picture}%
 \end{minipage}  
\begin{minipage}{.5\textwidth} 
\unitlength 0.1in
\begin{picture}(25,24)(0,-24) 
%
\special{pn 4}%
\special{pa 650 1200}%
\special{pa 2050 1200}%
\special{fp}%
%
\special{pn 4}%
\special{pa 2450 0}%
\special{pa 1250 2400}%
\special{fp}%
%
\special{pn 4}%
\special{pa 1650 1200}%
\special{pa 3050 1200}%
\special{fp}%
%
\special{pn 13}%
\special{ar 2250 400 50 50  0.0000000 6.2831853}%
%
\special{pn 13}%
\special{ar 650 2000 50 50  0.0000000 6.2831853}%
%
\special{pn 13}%
\special{ar 1850 1200 50 50  0.0000000 6.2831853}%
%
\special{pn 13}%
\special{ar 3050 400 50 50  0.0000000 6.2831853}%
%
\special{pn 13}%
\special{ar 2650 1200 50 50  0.0000000 6.2831853}%
%
\special{pn 13}%
\special{ar 1050 1200 50 50  0.0000000 6.2831853}%
%
\special{pn 13}%
\special{ar 1450 2000 50 50  0.0000000 6.2831853}%
%
\special{pn 13}%
\special{ar 2250 2000 50 50  0.0000000 6.2831853}%
%
\special{pn 13}%
\special{ar 1450 400 50 50  0.0000000 6.2831853}%
%
\special{pn 20}%
\special{pa 2150 1400}%
\special{pa 2048 1408}%
\special{fp}%
\special{pa 2150 1400}%
\special{pa 2064 1456}%
\special{fp}%
%
\special{pn 20}%
\special{pa 1850 1200}%
\special{pa 2250 400}%
\special{da 0.070}%
%
\special{pn 20}%
\special{pa 1850 1200}%
\special{pa 2650 1200}%
\special{da 0.070}%
%
\special{pn 20}%
\special{pa 1850 1200}%
\special{pa 650 2000}%
\special{da 0.070}%
%
\special{pn 20}%
\special{pa 1450 400}%
\special{pa 2050 0}%
\special{da 0.070}%
\special{pa 2050 0}%
\special{pa 2050 0}%
\special{da 0.070}%
%
\special{pn 20}%
\special{pa 3050 400}%
\special{pa 3650 0}%
\special{da 0.070}%
%
\special{pn 20}%
\special{pa 2250 2000}%
\special{pa 3150 1400}%
\special{da 0.070}%
%
\special{pn 20}%
\special{pa 1450 2000}%
\special{pa 1250 2400}%
\special{da 0.070}%
%
\special{pn 20}%
\special{pa 1050 1200}%
\special{pa 650 1200}%
\special{da 0.070}%
%
\special{pn 20}%
\special{pa 1650 2000}%
\special{pa 1750 2026}%
\special{fp}%
\special{pa 1650 2000}%
\special{pa 1750 1976}%
\special{fp}%
%
\special{pn 20}%
\special{pa 2350 1800}%
\special{pa 2418 1722}%
\special{fp}%
\special{pa 2350 1800}%
\special{pa 2372 1700}%
\special{fp}%
%
\special{pn 20}%
\special{pa 750 2000}%
\special{pa 1350 2000}%
\special{fp}%
%
\special{pn 20}%
\special{pa 1550 2000}%
\special{pa 2150 2000}%
\special{fp}%
%
\special{pn 20}%
\special{ar 2650 1200 100 100  0.0000000 6.2831853}%
%
\special{pn 20}%
\special{ar 3050 400 100 100  0.0000000 6.2831853}%
%
\special{pn 20}%
\special{ar 1450 2000 100 100  3.1415927 6.2831853}%
%
\special{pn 20}%
\special{ar 650 2000 100 100  5.9301949 6.2831853}%
%
\special{pn 20}%
\special{ar 2250 2000 102 102  3.1415927 5.1760366}%
%
\special{pn 20}%
\special{ar 2650 1200 200 200  5.1760366 6.2831853}%
\special{ar 2650 1200 200 200  0.0000000 2.0344439}%
%
\special{pn 20}%
\special{pa 746 1966}%
\special{pa 2556 1236}%
\special{fp}%
%
\special{pn 20}%
\special{pa 2296 1910}%
\special{pa 2560 1380}%
\special{fp}%
%
\special{pn 20}%
\special{pa 2740 1020}%
\special{pa 3006 490}%
\special{fp}%
\put(18.0000,-11.5000){\makebox(0,0)[rb]{$\omega _0$}}%
\put(29.4500,-3.0500){\makebox(0,0)[rb]{$\frac{\omega _0}{2}$}}%
\put(28.4000,-12.5000){\makebox(0,0)[lt]{$-\Xi _{\omega _2 ,q_{(012)}}$}}%
\put(16.2000,-20.5000){\makebox(0,0)[lt]{$-\Xi _{\omega _1 ,q_{(012)}}$}}%
\put(17.2000,-15.9500){\makebox(0,0)[lt]{$\Xi _{(01) ,q_{(012)}}$}}%
%
\special{pn 20}%
\special{pa 2840 800}%
\special{pa 2908 722}%
\special{fp}%
\special{pa 2840 800}%
\special{pa 2862 700}%
\special{fp}%
\put(28.2000,-7.7000){\makebox(0,0)[rb]{$\Delta _{(01)} ^{0,-1} (q _{(012)}) $}}%
%
\special{pn 4}%
\special{pa 2440 770}%
\special{pa 2770 830}%
\special{fp}%
\special{sh 1}%
\special{pa 2770 830}%
\special{pa 2708 798}%
\special{pa 2718 820}%
\special{pa 2702 838}%
\special{pa 2770 830}%
\special{fp}%
%
\special{pn 20}%
\special{pa 2996 490}%
\special{pa 2686 1110}%
\special{dt 0.054}%
\end{picture}%
 \end{minipage} 
\begin{minipage}{.5\textwidth} 
\unitlength 0.1in
\begin{picture}( 30.5000, 24.0000)(  6.0000,-24.0000)
%
\special{pn 4}%
\special{pa 650 1200}%
\special{pa 2050 1200}%
\special{fp}%
%
\special{pn 4}%
\special{pa 2450 0}%
\special{pa 1250 2400}%
\special{fp}%
%
\special{pn 4}%
\special{pa 1650 1200}%
\special{pa 3050 1200}%
\special{fp}%
%
\special{pn 13}%
\special{ar 2250 400 50 50  0.0000000 6.2831853}%
%
\special{pn 13}%
\special{ar 650 2000 50 50  0.0000000 6.2831853}%
%
\special{pn 13}%
\special{ar 1850 1200 50 50  0.0000000 6.2831853}%
%
\special{pn 13}%
\special{ar 3050 400 50 50  0.0000000 6.2831853}%
%
\special{pn 13}%
\special{ar 2650 1200 50 50  0.0000000 6.2831853}%
%
\special{pn 13}%
\special{ar 1050 1200 50 50  0.0000000 6.2831853}%
%
\special{pn 13}%
\special{ar 1450 2000 50 50  0.0000000 6.2831853}%
%
\special{pn 13}%
\special{ar 2250 2000 50 50  0.0000000 6.2831853}%
%
\special{pn 13}%
\special{ar 1450 400 50 50  0.0000000 6.2831853}%
%
\special{pn 20}%
\special{pa 2150 1400}%
\special{pa 2048 1408}%
\special{fp}%
\special{pa 2150 1400}%
\special{pa 2064 1456}%
\special{fp}%
%
\special{pn 20}%
\special{pa 1850 1200}%
\special{pa 2250 400}%
\special{da 0.070}%
%
\special{pn 20}%
\special{pa 1850 1200}%
\special{pa 2650 1200}%
\special{da 0.070}%
%
\special{pn 20}%
\special{pa 1850 1200}%
\special{pa 650 2000}%
\special{da 0.070}%
%
\special{pn 20}%
\special{pa 1450 400}%
\special{pa 2050 0}%
\special{da 0.070}%
\special{pa 2050 0}%
\special{pa 2050 0}%
\special{da 0.070}%
%
\special{pn 20}%
\special{pa 3050 400}%
\special{pa 3650 0}%
\special{da 0.070}%
%
\special{pn 20}%
\special{pa 2250 2000}%
\special{pa 3150 1400}%
\special{da 0.070}%
%
\special{pn 20}%
\special{pa 1450 2000}%
\special{pa 1250 2400}%
\special{da 0.070}%
%
\special{pn 20}%
\special{pa 1050 1200}%
\special{pa 650 1200}%
\special{da 0.070}%
%
\special{pn 20}%
\special{pa 1650 2000}%
\special{pa 1750 2026}%
\special{fp}%
\special{pa 1650 2000}%
\special{pa 1750 1976}%
\special{fp}%
%
\special{pn 20}%
\special{pa 2550 1400}%
\special{pa 2484 1478}%
\special{fp}%
\special{pa 2550 1400}%
\special{pa 2528 1500}%
\special{fp}%
%
\special{pn 20}%
\special{pa 750 2000}%
\special{pa 1350 2000}%
\special{fp}%
%
\special{pn 20}%
\special{pa 1550 2000}%
\special{pa 2150 2000}%
\special{fp}%
%
\special{pn 20}%
\special{ar 2650 1200 100 100  0.0000000 6.2831853}%
%
\special{pn 20}%
\special{ar 1450 2000 100 100  3.1415927 6.2831853}%
%
\special{pn 20}%
\special{ar 650 2000 100 100  5.9301949 6.2831853}%
%
\special{pn 20}%
\special{pa 746 1966}%
\special{pa 2556 1236}%
\special{fp}%
\put(18.0000,-11.5000){\makebox(0,0)[rb]{$-\omega _1$}}%
\put(23.2000,-20.7000){\makebox(0,0)[lt]{$\frac{\omega _0}{2}$}}%
\put(16.2000,-20.5000){\makebox(0,0)[lt]{$-\Xi _{\omega _1 ,q_{(012)}}$}}%
\put(17.2000,-15.9500){\makebox(0,0)[lt]{$\Xi _{(01) ,q_{(012)}}$}}%
\put(30.2000,-11.7000){\makebox(0,0)[lb]{$\Delta _{(01)} ^{0,1} (q _{(012)}) $}}%
%
\special{pn 20}%
\special{ar 2250 2000 100 100  0.0000000 6.2831853}%
%
\special{pn 20}%
\special{pa 2606 1290}%
\special{pa 2296 1910}%
\special{dt 0.054}%
%
\special{pn 4}%
\special{pa 3020 1170}%
\special{pa 2570 1450}%
\special{fp}%
\special{sh 1}%
\special{pa 2570 1450}%
\special{pa 2638 1432}%
\special{pa 2616 1422}%
\special{pa 2616 1398}%
\special{pa 2570 1450}%
\special{fp}%
\end{picture}%
 \end{minipage} 
\end{proof} 
We have the common basis for 
$\mathcal{H} ( q_{(012)})$ and 
$\mathcal{H} \left( q_{( \tau _{ij} (0), \tau _{ij} (1) , \tau _{ij} (2) )} \right)$ : 
Put 
\begin{align*} 
\Xi _{(kl)} &:=  
  \frac{1}{c -1 } \overline{S _{\epsilon} \left( \frac{\omega _k}{2} \right)} 
\otimes 
\varsigma _{\overline{S _{\epsilon} \left( \frac{\omega _k}{2} \right)} } 
+ \overline{\left[ \frac{\omega _k}{2} +\epsilon , \frac{\omega _l}{2}
 -\epsilon \right]} 
\otimes 
\varsigma _{\overline{\left[ \frac{\omega _k}{2} +\epsilon , 
\frac{\omega _l}{2} -\epsilon \right]}} 
- \frac{1}{c -1 } \overline{S _{\epsilon} \left( \frac{\omega _l}{2} \right)} 
\otimes 
\varsigma _{\overline{S _{\epsilon} \left( \frac{\omega _l}{2} \right)} } \\ 
 \Xi _{\omega _k} &:=  
 l _{\omega _k} \otimes \varsigma _{l _{\omega _k}} , 
\end{align*} 
and we define the ordering on the index set 
$J ^{\prime}:= \{ (01) , (20) , \omega _1 , \omega _2 \}$ 
of  this basis :  
\begin{align*} 
 (01) \prec  (20) \prec \omega _1 \prec \omega _2 .  
\end{align*} 
Now we have the following : 
\begin{thm}[connection matrices]  \label{thm_connection_matrix} 
Let $M ( \gamma _{(ij)} ^{m_1 , m_2}) = 
\left[ m ( \gamma _{(ij)} ^{m_1 , m_2}) _{\mu ,\nu} \right] 
_{\mu , \nu \in J^{\prime}}$ 
be the $4 \times 4$-matrix representing the linear isomorphism 
$\left( \gamma _{(ij)} ^{m_1 , m_2} \right)_{\ast}$, that is,  
$\left( \gamma _{(ij)} ^{m_1 , m_2} \right)_{\ast} 
( \Xi _{\mu } ) = \sum _{\nu} 
m (\gamma_{(ij)} ^{m_1 , m_2} ) _{\mu , \nu } 
\Xi _{\mu } 
$. Then, 
\begin{align*} 
M (\gamma _{(02)} ^{-1,0})
 &=  
\begin{bmatrix} 
1 &  0 & 0   & 0    \\
c & -c &c-1 &-c+1  \\ 
c & -c-1&c    & -c+1  \\
c &-c-1&c-1 &-c+2
\end{bmatrix}  ,   
& M (\gamma _{(02)} ^{1,0}) 
=  &
\begin{bmatrix} 
1 &  0 & 0   & 0    \\
c & -c & c-1   &0  \\ 
0 & 0 & 1   & 0  \\
c &  -c-1 & c-1   & 1
\end{bmatrix} , \\
M (\gamma _{(01)} ^{0,-1}) 
&=  
\begin{bmatrix} 
-c &  1 & -c+1   & c-1    \\
0  & 1  & 0   & 0  \\ 
c+1 & -1& c  & -c+1  \\
c+1 &  -1 & c-1   & -c+2
\end{bmatrix} , 
& M (\gamma _{(01)} ^{0,1})
= &
\begin{bmatrix} 
-c &  1 & 0   & c-1    \\
0  & 1  & 0   & 0  \\ 
c+1  & -1  & 1  & -c+1  \\
0 &  0 & 0   & 1
\end{bmatrix} .    
\end{align*} 
\end{thm} 
\begin{proof} 
Substituting the result of Proposition {\ref{vanishing_cycle}} 
for the twisted Picard-Lefschetz formula 
\begin{align*} 
 \left( \gamma _{(ij)} ^{m_1 , m_2} \right)_{\ast} 
( \Xi ) 
= 
\Xi + \frac{-c-1}{\left\langle \Delta _{(ij)} ^{m_1 , m_2} , 
{\Delta _{(ij)} ^{m_1 , m_2} }^{\vee} \right\rangle} 
\left\langle \Xi , {\Delta _{(ij)} ^{m_1 , m_2} }^{\vee} \right\rangle 
\Delta _{(ij)} ^{m_1 , m_2} ,  
\end{align*}  
we obtain the assertion. 
\end{proof}



{\scshape
\begin{flushright}
\begin{tabular}{l}
Department of Mathematics \\ 
Faculty of Science \\ 
 Kyoto University \\ 
Kitashirakawa-Oiwakechou, \\
Sakyouku, Kyoto 606-8502, \\ 
 Japan \\ 
{\upshape e-mail: koki@math.kyoto-u.ac.jp}\\\\
\end{tabular}
\end{flushright}
}


\end{document}